\documentclass{conm-p-l}
\usepackage{amssymb}

\usepackage{amsfonts}
\usepackage{eucal}
\begin{document}

\font\eightrm=cmr8
\font\eightit=cmti8
\font\eighttt=cmtt8
\def\tci
{\hbox{\hskip1.8pt$\rightarrow$\hskip-11.5pt$^{^{C^\infty}}$\hskip-1.3pt}}
\def\nft
{\hbox{$n$\hskip3pt$\equiv$\hskip4pt$5$\hskip4.4pt$($mod\hskip2pt$3)$}}
\def\bbR{\mathrm{I\!R}}
\def\rto{\bbR\hskip-.5pt^2}
\def\rtr{\bbR\hskip-.7pt^3}
\def\rfo{\bbR\hskip-.7pt^4}
\def\rn{\bbR^{\hskip-.6ptn}}
\def\mr{\bbR^{\hskip-.6ptm}}
\def\bbZ{\mathsf{Z\hskip-4ptZ}}
\def\bbRP{\text{\bf R}\text{\rm P}}
\def\bbC{\text{\bf C}}
\def\hyp{\hskip.5pt\vbox
{\hbox{\vrule width3ptheight0.5ptdepth0pt}\vskip2.2pt}\hskip.5pt}
\def\er{r}
\def\es{s}
\def\df{d\hskip-.8ptf}
\def\fv{\mathcal{F}}
\def\fvp{\fv_{\nrmh p}}
\def\ns{\nabla\hskip-2.7pt_}
\def\wv{\mathcal{W}}
\def\vt{\mathcal{P}}
\def\dz{\mathcal{P}}
\def\dzx{\mathcal{P}\hskip-2pt_x\w}
\def\vz{\mathcal{V}}
\def\hz{\mathcal{H}}
\def\ko{\mathcal{L}}
\def\vx{\vz\hskip-2.3pt_x\w}
\def\tv{\mathcal{T}}
\def\vtx{\vt_{\nh x}}
\def\fh{f}
\def\rc{\theta}
\def\jm{\mathcal{I}}
\def\ke{\mathcal{K}}
\def\xc{\mathcal{X}_c}
\def\lz{\mathcal{L}}
\def\tzq{T\hskip-2.7pt_z\w\nh\qe}
\def\tyhm{{T\hskip-3.5pt_y\w\hm}}
\def\tlm{{T\hskip-3.5pt_L\w\nnh M}}
\def\tm{T\nh\mf}
\def\Lie{\pounds}
\def\lv{\Lie\hskip-1.2pt_v\w}
\def\lo{\lz_0}
\def\xe{\mathcal{E}}
\def\eo{\xe_0}
\def\hm{\hskip1.9pt\widehat{\hskip-1.9ptM\hskip-.2pt}\hskip.2pt}
\def\hmt{\hskip1.9pt\widehat{\hskip-1.9ptM\hskip-.5pt}_t}
\def\hmz{\hskip1.9pt\widehat{\hskip-1.9ptM\hskip-.5pt}_0}
\def\hmp{\hskip1.9pt\widehat{\hskip-1.9ptM\hskip-.5pt}_p}
\def\hg{\hskip1.2pt\widehat{\hskip-1.2ptg\hskip-.4pt}\hskip.4pt}
\def\hk{\hskip1.5pt\widehat{\hskip-1.5ptK\hskip-.5pt}\hskip.5pt}
\def\hq{\hskip1.5pt\widehat{\hskip-1.5ptQ\hskip-.5pt}\hskip.5pt}
\def\q{s}
\def\qe{Q}
\def\bq{\hat q}
\def\p{p}
\def\w{^{\phantom i}}
\def\x{v}
\def\y{y}
\def\vp{\vt^\perp}
\def\vd{\vt\hh'}
\def\vdx{\vd{}\hskip-4.5pt_x}
\def\bz{b\hh}
\def\fe{F}
\def\fy{\phi}
\def\vl{\Lambda}
\def\stf{\lambda}
\def\hy{\mathcal{V}}
\def\vh{h}
\def\mv{V}
\def\vo{V_{\nnh0}}
\def\ao{A_0}
\def\bo{B_0}
\def\uv{\mathcal{U}}
\def\sv{\mathcal{S}}
\def\svp{\sv_p}
\def\xv{\mathcal{X}}
\def\xvp{\xv_p}
\def\yv{\mathcal{Y}}
\def\yvp{\yv_p}
\def\zv{\mathcal{Z}}
\def\zvp{\zv_p}
\def\cv{\mathcal{C}}
\def\dy{\mathcal{D}}
\def\nv{\mathcal{N}}
\def\iv{\mathcal{I}}
\def\gkp{\Sigma}
\def\ret{\pi}
\def\hs{\hskip.7pt}
\def\hh{\hskip.4pt}
\def\hn{\hskip-.4pt}
\def\nh{\hskip-.7pt}
\def\nnh{\hskip-1pt}
\def\hrz{^{\hskip.5pt\text{\rm hrz}}}
\def\vrt{^{\hskip.2pt\text{\rm vrt}}}
\def\vt{\varTheta}
\def\op{S}
\def\vg{\varGamma}
\def\bvg{\hskip3pt\overline{\hskip-3pt\vg\nh}\hs}
\def\ny{\nu}
\def\gy{\lambda}
\def\ax{\alpha}
\def\of{\omega}
\def\bx{\beta}
\def\cx{\gamma}
\def\ay{a}
\def\by{b}
\def\cy{c}
\def\gp{\mathrm{G}}
\def\hp{\mathrm{H}}
\def\kp{\mathrm{K}}
\def\gm{\gamma}
\def\Gm{\Gamma}
\def\Lm{\Lambda}
\def\Dt{\Delta}
\def\dg{\Delta}
\def\sj{\sigma}
\def\lg{\langle}
\def\rg{\rangle}
\def\lr{\lg\hh\cdot\hs,\hn\cdot\hh\rg}
\def\vs{vector space}
\def\rvs{real vector space}
\def\vf{vector field}
\def\tf{tensor field}
\def\tvn{the vertical distribution}
\def\dn{distribution}
\def\pt{point}
\def\tc{tor\-sion\-free connection}
\def\ea{equi\-af\-fine}
\def\rt{Ric\-ci tensor}
\def\pde{partial differential equation}
\def\pf{projectively flat}
\def\pfs{projectively flat surface}
\def\pfc{projectively flat connection}
\def\pftc{projectively flat tor\-sion\-free connection}
\def\su{surface}
\def\sco{simply connected}
\def\psr{pseu\-\hbox{do\hs-}Riem\-ann\-i\-an}
\def\inv{-in\-var\-i\-ant}
\def\trinv{trans\-la\-tion\inv}
\def\feo{dif\-feo\-mor\-phism}
\def\feic{dif\-feo\-mor\-phic}
\def\feicly{dif\-feo\-mor\-phi\-cal\-ly}
\def\Feicly{Dif\-feo\-mor\-phi\-cal\-ly}
\def\diml{-di\-men\-sion\-al}
\def\prl{-par\-al\-lel}
\def\skc{skew-sym\-met\-ric}
\def\sky{skew-sym\-me\-try}
\def\Sky{Skew-sym\-me\-try}
\def\dbly{-dif\-fer\-en\-ti\-a\-bly}
\def\cs{con\-for\-mal\-ly symmetric}
\def\cf{con\-for\-mal\-ly flat}
\def\ls{locally symmetric}
\def\ecs{essentially con\-for\-mal\-ly symmetric}
\def\rr{Ric\-ci-re\-cur\-rent}
\def\kf{Killing field}
\def\om{\omega}
\def\vol{\varOmega}
\def\dv{\delta}
\def\ve{\varepsilon}
\def\zt{\zeta}
\def\kx{\kappa}
\def\mf{M}
\def\my{\mf\hskip-2.3pt_y\w}
\def\bs{\varSigma}
\def\tab{{T\hskip-.2pt^*\hskip-2.3pt\bs}} 
\def\tayb{{T_{\hskip-1.4pty}^*\hskip-2pt\bs}}
\def\tazq{{T_{\hskip-1.4ptz}^*\hskip-2pt\qe}} 
\def\taxb{{T_x^*\hskip-1pt\bs}} 
\def\tyb{{T\hskip-2pt_y\w\hskip-.2pt\bs}} 
\def\txm{{T\hskip-3pt_x\w\hskip-.5ptM}} 
\def\tym{{T\hskip-2pt_y\w\hskip-.9ptM}} 
\def\mfd{-man\-i\-fold}
\def\bmf{base manifold}
\def\bd{bundle}
\def\qs{\mathrm{q}}
\def\mq{\mathrm{p}}
\def\tbd{tangent bundle}
\def\ctb{cotangent bundle}
\def\bp{bundle projection}
\def\prc{pseu\-\hbox{do\hs-}Riem\-ann\-i\-an metric}
\def\prd{pseu\-\hbox{do\hs-}Riem\-ann\-i\-an manifold}
\def\Prd{pseu\-\hbox{do\hs-}Riem\-ann\-i\-an manifold}
\def\npd{null parallel distribution}
\def\pj{-pro\-ject\-a\-ble}
\def\pd{-pro\-ject\-ed}
\def\lcc{Le\-vi-Ci\-vi\-ta connection}
\def\vb{vector bundle}
\def\vbm{vec\-tor-bun\-dle morphism}
\def\kerd{\text{\rm Ker}\hskip2.7ptd}
\def\ro{\rho}
\def\sy{\sigma}
\def\ts{total space}
\def\pmb{\pi}

\newtheorem{theorem}{Theorem}[section] 
\newtheorem{proposition}[theorem]{Proposition} 
\newtheorem{lemma}[theorem]{Lemma} 
\newtheorem{corollary}[theorem]{Corollary} 
  
\theoremstyle{definition} 
  
\newtheorem{defn}[theorem]{Definition} 
\newtheorem{notation}[theorem]{Notation} 
\newtheorem{example}[theorem]{Example} 
\newtheorem{conj}[theorem]{Conjecture} 
\newtheorem{prob}[theorem]{Problem} 
  
\theoremstyle{remark} 
  
\newtheorem{remark}[theorem]{Remark}

\renewcommand{\theequation}{\arabic{section}.\arabic{equation}}

\title[Null-pro\-ject\-a\-bil\-i\-ty of Le\-vi-Ci\-vi\-ta
connections]{Null-pro\-ject\-a\-bil\-i\-ty of Le\-vi-Ci\-vi\-ta connections}
\author[A. Derdzinski]{Andrzej Derdzinski} 
\address{Department of Mathematics, The Ohio State University, 
Columbus, OH 43210} 
\email{andrzej@math.ohio-state.edu} 
\author[K.\ Masood]{Kirollos Masood} 
\address{Department of Mathematics, The Ohio State University, 
Columbus, OH 43210} 
\email{masood.24@buckeyemail.osu.edu} 
\def\leftmark{A.\ Derdzinski \&\ K.\ Masood}
\def\rightmark{Null-pro\-ject\-a\-bil\-i\-ty of Le\-vi-Ci\-vi\-ta connections}

\begin{abstract}
We study the natural property of pro\-ject\-a\-bil\-i\-ty of a tor\-sion-free
connection along a foliation on the underlying manifold, which leads to a
projected tor\-sion-free connection on a local leaf space, focusing on
pro\-ject\-a\-bil\-i\-ty of Le\-vi-Ci\-vi\-ta connections of 
pseu\-\hbox{do\hs-}Riem\-ann\-i\-an metric along foliations tangent to null 
parallel distributions. For the neutral metric signature and
mid-di\-men\-sion\-al distributions, Afifi showed in 1954 that
pro\-ject\-a\-bil\-i\-ty of the Le\-vi-Ci\-vi\-ta connection characterizes,
locally, the case of Pat\-ter\-son and Walk\-er's Riemann extension metrics. We
extend this correspondence to null parallel distributions of any dimension,
introducing a suitable generalization of Riemann extensions.
\end{abstract}

\maketitle

\setcounter{section}{0}
\setcounter{theorem}{0}
\renewcommand{\thetheorem}{\Alph{theorem}}
\section*{Introduction}
\setcounter{equation}{0}
Pro\-ject\-a\-ble connections, within the more general
context of transverse geometry for foliations, have been studied by many
authors; see, for instance,
\cite{bednarska,bel'ko, molino,vaisman,zhukova-dolgonosova}. The present paper
focuses on certain special connections and foliations arising in
pseu\-\hbox{do\hs-}Riem\-ann\-i\-an geometry.

Given an integrable distribution $\,\vz\hs$ and a tor\-sion-free connection
$\,\nabla\hs$ on a manifold, $\,\vz\nh$-pro\-ject\-a\-bil\-i\-ty of
$\,\nabla\hs$ onto a connection $\,D\,$ on a local leaf space means that,
whenever local vector field $\,v,w\,$ are $\,\vz\nh$-pro\-ject\-a\-ble, so
must be $\,\ns v\w w$. This is equivalent (Section~\ref{pc}) to requiring the
horizontal distribution of $\,\nabla\hs$ to project onto that of $\,D$.
In the case of a null parallel distribution $\,\vz\hs$ on a
pseu\-\hbox{do\hs-}Riem\-ann\-i\-an manifold with the Le\-vi-Ci\-vi\-ta
connection $\,\nabla\nh$, Walk\-er's theorem \cite{walker} easily leads to a
lo\-cal-co\-or\-di\-nate characterization (Lemma~\ref{nbprj}) of
$\,\vz\nh$-pro\-ject\-a\-bil\-i\-ty of $\,\nabla\nh$.

A less trivial aspect of the latter situation involves the resulting geometric
invariants, not manifestly present in the lo\-cal-co\-or\-di\-nate
description. For the neutral metric signature and 
mid-di\-men\-sion\-al distributions, the underlying manifold $\,\mf$ forms,
locally, a cotangent af\-fine bundle (Section~\ref{ca}) over the leaf space
$\,\bs\,$ (endowed with a tor\-sion-free connection $\,D$) and a result due to
Afifi \cite[p.\ 313]{afifi}, which we reproduce
as Theorem~\ref{afifi}, states that the metric $\,g\,$ is a
Pat\-ter\-son\hs-\nnh Walk\-er Riemann extension for $\,D$.

In Section~\ref{rp} we extend Afifi's theorem to arbitrary indefinite metric
signatures and null
parallel distributions $\,\vz\hs$ of all dimensions. Rather than being just
a cotangent af\-fine bundle, $\,\mf\,$ is now, locally, a bundle of
co\-tan\-gent-prin\-ci\-pal bundles (Section~\ref{ca}) over $\,\bs$, meaning
that there is a bundle $\,\qe\,$ over $\,\bs$, and $\,\mf\,$ itself
constitutes an af\-fine bundle over the total space of $\,\qe\,$ with the
pull\-back of $\,\tab\,$ to $\,\qe$ serving as its associated vector bundle.
The relevant geometric invariants include, in addition to the
tor\-sion-free connection
$\,D\,$ on $\,\bs$, also a vertical metric $\,\vh\,$ on $\,\qe$ (in other
words, a pseu\-\hbox{do\hs-}Riem\-ann\-i\-an fibre metric in the vertical
distribution of the bundle projection $\,\qe\to\bs$, which makes $\,\qe\,$ a
bundle of pseu\-\hbox{do\hs-}Riem\-ann\-i\-an manifolds over $\,\bs$). The
original metric $\,g\,$ now constitutes what we call a {\it Riemann
pull\-back-ex\-ten\-sion\/} for $\,D,\qe\,$ and $\,\vh$.

The mid-di\-men\-sion\-al situation is a special case of this picture, with 
$\,\qe=\bs$, so that the bundle $\,\qe\,$ has sin\-gle-point fibres and
$\,\vh=0$.

\renewcommand{\thetheorem}{\thesection.\arabic{theorem}}
\section{Preliminaries}\label{pr} 
\setcounter{equation}{0}
We always assume $\,C^\infty\nnh$-dif\-fer\-en\-ti\-a\-bil\-i\-ty of
manifolds, mappings (including bundle projections), sub\-bundles (such as
distributions), and tensor fields. Manifolds are by definition connected. 
Maximal connected integral manifolds of foliations
(in\-te\-gra\-ble distributions) are referred to as their {\it leaves}.

Our sign convention for the curvature tensor $\,R\,$ of a connection 
$\,\nabla\hs$ on a manifold $\,\mf\nh$, and any tangent vector fields 
$\,v,u,u'\nh$, is
\begin{equation}\label{cur}
R(v,u)\hh u'\,\,=\,\,\hs\ns u\w\ns v\w u'\,-\,\,\ns v\w\ns u\w u'\,
+\,\,\ns{\,[v,u]}\w u',
\end{equation}
the coordinate form of which reads
$\,R_{ijk}\w{}^l\nh=\partial\hskip-1.5pt_j\w\vg_{\hskip-2.6ptik}^{\hs l}
-\partial\nnh_i\w\vg_{\hskip-2.6ptjk}^{\hs l}
+\vg_{\hskip-2.6ptjp}^{\hs l}\vg_{\hskip-2.6ptik}^{\hs p}
-\vg_{\hskip-2.6ptip}^{\hs l}\vg_{\hskip-2.6ptjk}^{\hs p}$.

Given a mapping $\,\pi:\mf\nh\to\bs\,$ between manifolds, a vector 
field $\,w\,$ (or, a distribution $\,\vz$) on $\,\mf\,$ is said to be 
$\,\pi${\it-pro\-ject\-a\-ble\/} if $\,d\pi\nnh_x\w w_x\w=\hs u_{\pi(x)}\w$,
or
\begin{equation}\label{dpv}
d\pi\nnh_x\w(\vz\nnh_x\w)\,=\,\mathcal{W}\nnh_{\pi(x)}\w\hskip4pt\mathrm{\
for\ all\ }\,x\in\mf
\end{equation}
and some vector field $\,u\,$ (or, respectively, some distribution
$\,\mathcal{W}$) on $\,\bs$.

Let $\,\pi:\mf\nh\to\bs\,$ be a bundle projection with the vertical distribution 
$\,\vz=\,\mathrm{Ker}\hskip2.3ptd\pi$, and let $\,w\,$ be a vector field on
$\,\mf\nh$. Then
\begin{equation}\label{prj}
\begin{array}{l}
w\hs\mathrm{\ is\ }\pi\hyp\mathrm{pro\-ject\-a\-ble\ if\ and\ only\ if,\ 
for\ every\ section\ }v\\
\mathrm{of\ }\hs\vz\nh\mathrm{\,\ the\ Lie\ 
bracket\ }\hs[v,\nh w]\hs\mathrm{\ is\ also\ a\ section\ of\
}\hs\vz\nh\mathrm{,\ or,}\\
\mathrm{equivalently,\ the\ local\ flow\ of\ }\,\hs w\hs\,\mathrm{\ leaves\ 
}\,\hs\vz\,\mathrm{\ invariant.}
\end{array}
\end{equation}
(This is obvious if one uses local 
coordinates for $\,\mf\,$ making $\,\pi\,$ appear as a 
Car\-te\-sian-prod\-uct projection.) Given an integrable distribution
$\,\vz\hs$ 
on a manifold $\,\mf\nh$, 
every point of $\,\mf\,$ has a neighborhood $\,\,U$ such that, for some 
manifold $\,\bs$,
\begin{equation}\label{fbr}
\begin{array}{l}
\mathrm{the\ leaves\ of\ }\,\vz\,\mathrm{\ restricted\ to\ 
}\,\,U\hs\mathrm{\ are\ the}\\
\mathrm{fibres\ of\ a\ bundle\ projection\ }\,\pi\nnh:\nh U\to\bs\hh.
\end{array}
\end{equation}
  
\section{Projectable connections}\label{pc}
\setcounter{equation}{0}
Suppose that $\,\vz\hs$ is an integrable distribution 
on a manifold $\,\mf\nh$. By $\,\vz\nh${\it-pro\-ject\-a\-bil\-i\-ty\/} of a
vector field $\,w\,$ on
an open set $\,\,U'\subseteq\mf\,$ we mean its
$\,\pi$-pro\-ject\-a\-bil\-i\-ty for any $\,\pi,\hs U\nh,\bs\,$ with
(\ref{fbr}) such that $\,\,U\subseteq U'\nh$.

Let $\,\vz\hs$ be an integrable distribution on a manifold $\,\mf\,$
equipped with a tor\-sion-free connection $\,\nabla\nh$. We say that
$\,\nabla\hs$ is $\,\vz\nh${\it-pro\-ject\-a\-ble\/} (or
{\it pro\-ject\-a\-ble along\/} $\,\vz$) if, for any 
$\,\vz\nh$-pro\-ject\-a\-ble vector fields $\,v,w\,$ on an open subset
of $\,\mf\nh$, the covariant derivative $\,\ns v\w w\,$ is 
$\,\vz\nh$-pro\-ject\-a\-ble as well.

For an integrable distribution $\,\vz\hs$ on a manifold $\,\mf\,$ and the
restriction $\,\tlm\,$ of the tangent bundle $\,\tm\,$ to any given leaf
$\,L\,$ of $\,\vz\nh$, the normal bundle of $\,L\hh$ in $\,\mf\,$ is given by 
$\,N\nh L\nh=\hs\tlm\nh/\hs T\nh L$, with the quo\-tient-bun\-dle projection
$\,\tlm\to N\nh L$.
\begin{lemma}\label{prprl}Let a tor\-sion-free connection\/ 
$\,\nabla\hs$ on a manifold\/ $\,\mf\,$ be pro\-ject\-a\-ble along an
integrable distribution\/ $\,\vz\nh$. Then, with\/ $\,L\hh$ denoting any
given leaf of\/ $\,\vz\nh$,
\begin{enumerate}
  \def\theenumi{{\rm\roman{enumi}}}
\item both\/ $\,\ns w\w v,\,\,\ns v\w w\,$ are local sections of\/
$\,\vz\nh$, for any\/ $\,\vz\nh$-pro\-ject\-a\-ble local vector
field\/ $\,w\,$ in\/ $\,\mf\,$ and any local section\/ $\,v\,$ of\/
$\,\vz\hs$ with the same domain,
\item $\vz\,$ is\/ $\,\nabla\nh$-par\-al\-lel, and so, in particular,
 $\,\nabla\hs$ induces a connection in the normal bundle\/ 
$\,N\nh L=\tlm/T\nh L\,$ of\/ $\,L$,
\item  the image under the quo\-tient-bun\-dle projection\/
$\,\tlm\to N\nh L\,$ of any\/ $\,\vz\nh$-pro\-ject\-a\-ble vector field\/
$\,w\,$ on an open set\/ $\,\,U\subseteq\mf\,$ intersecting\/ $\,L\,$ is a
local section of $\,N\nh L$, parallel relative to the connection induced by\/
$\,\nabla\nh$.
\end{enumerate}
\end{lemma}
\begin{proof}Given $\,w,v\,$ as in (i),
$\,\ns v\w w\,$ is $\,\vz\nh$-pro\-ject\-a\-ble (as $\,v\,$ is). This remains
the case after $\,v\,$ has been multiplied by any function. Thus, 
$\,\ns v\w w\,$ projects onto $\,0\,$ (or else, multiplied by a 
non-pro\-ject\-a\-ble function, it would cease to be
$\,\vz\nh$-pro\-ject\-a\-ble); in other words, $\,\ns v\w w\,$ is a section
of $\,\vz\nh$. Due to (\ref{prj}), $\,\ns w\w v\,$ differs from
$\,\ns v\w w\,$ by the section $\,[w,v]\,$ of $\,\vz\nh$, and (i) follows,
also proving (ii) -- (iii).
\end{proof}
\begin{lemma}\label{prjcn}Whenever a tor\-sion-free connection\/ 
$\,\nabla\hs$ on a manifold\/ $\,\mf\,$ is pro\-ject\-a\-ble along an
integrable distribution\/ $\,\vz\nh$, it gives rise to a 
tor\-sion-free connection\/ $\,D\,$ on each local leaf space\/ $\,\bs\,$
with\/ {\rm(\ref{fbr})}, characterized by\/ $\,D\nnh_v\w\hn w=\ns v\w w\,$
for\/ $\,\vz\nh$-pro\-ject\-a\-ble vector fields\/ $\,v,w\,$ on the open set\/ 
$\,\,U\subseteq\mf\,$ appearing in\/ {\rm(\ref{fbr})}, where the same
symbols\/ $\,v,w\,$ denote their projections onto\/ $\,\bs$.
\end{lemma}
In fact, $\,D\,$ is well defined: replacing $\,v,w\,$ by other vector fields
on $\,\,U\,$ having the same projections onto $\,\bs$, that is, adding to
them local sections of $\,\vz\nh$, results -- by Lemma~\ref{prprl}(i) --
only in adding to $\,\ns v\w w\,$ a local section of $\,\vz\nh$, without
changing the projection of $\,\ns v\w w\,$ onto $\,\bs$.
\begin{lemma}\label{lccoo}In local coordinates\/
$\,x\hh^1\nh,\dots,x\hh^n\nh$, let the\/ $\,\q$-di\-men\-sion\-al 
distribution\/ $\,\vz\hs$ be spanned by the last\/ $\,\q\,$ coordinate 
vector fields. Then a tor\-sion-free connection\/ $\,\nabla\hs$ is\/
$\,\vz\nh$-pro\-ject\-a\-ble if and only if
its component functions satisfy the relations\/
$\,\vg_{\hskip-2.6ptaj}^{\hs i}\nh
=\nh\vg_{\hskip-2.6ptab}^{\hs i}\nh
=\nh\partial\nnh_a\w\vg_{\hskip-2.6ptjk}^{\hs i}\nh=0\,$ for all\/ 
$\,i,j,k\in\{1,\dots,n-\q\}\,$ and\/ $\,a,b\in\{n-\q+1,\dots,n\}$.

In the case of\/ $\,\vz\nh$-pro\-ject\-a\-bil\-i\-ty of\/ $\,\nabla\nh$,
the projected tor\-sion-free connection\/ $\hs D\hs$ on a local leaf space\/ 
$\,\bs\,$ with the coordinates\/ $\,x\hh^i$ for\/
$\,i=1,\dots,n-\q$, cf.\ Lemma\/~{\rm\ref{prjcn}}, has the component
functions\/ $\,\vg_{\hskip-2.6ptjk}^{\hs i}$.
\end{lemma}
\begin{proof}First, let $\,\nabla\hs$ be $\,\vz\nh$-pro\-ject\-a\-ble.
By Lemma~\ref{prprl}(ii), $\,\vz\,$ is\/ $\,\nabla\nh$-par\-al\-lel, so that
$\,\vg_{\hskip-2.6ptaj}^{\hs i}\nh
=\nh\vg_{\hskip-2.6ptab}^{\hs i}\nh=0$, while
$\,\vz\nh$-pro\-ject\-a\-bil\-i\-ty of the coordinate vector fields
$\,\partial\nnh_i\w$ implies the same for
$\,\ns{\partial\nnh_j}\w\nh\partial\nnh_k\w
=\vg_{\hskip-2.6ptjk}^{\hs i}\partial\nnh_i\w
+\vg_{\hskip-2.6ptjk}^{\hs a}\partial\nnh_a\w$, that is,
for $\,\vg_{\hskip-2.6ptjk}^{\hs i}\partial\nnh_i\w$ (the last term
being a local section $\,v\,$ of $\,\vz\nh$, and hence
$\,\vz\nh$-pro\-ject\-a\-ble). The functions
$\,\vg_{\hskip-2.6ptjk}^{\hs i}$ are thus constant along $\,\vz\nh$, 
and so $\,\partial\nnh_a\w\vg_{\hskip-2.6ptjk}^{\hs i}\nh=0$.

Assume now that $\,\vg_{\hskip-2.6ptaj}^{\hs i}\nh
=\nh\vg_{\hskip-2.6ptab}^{\hs i}\nh
=\nh\partial\nnh_a\w\vg_{\hskip-2.6ptjk}^{\hs i}\nh=0\,$ with 
index ranges as above. Due to $\,\vz\nh$-pro\-ject\-a\-bil\-i\-ty of the
coordinate vector fields, any $\,\vz\nh$-pro\-ject\-a\-ble local vector field
has the form $\,w=w^i\partial\nnh_i\w+w^a\partial\nnh_a\w$ with
$\,\partial\nnh_a\w w^i\nh=0$, so that, for two such vector fields
$\,w,u$, since $\,\vz\hs$ is clearly $\,\nabla\nh$-par\-al\-lel, 
$\,\ns u\w w\,$ equals a local section of $\,\vz\hs$ plus
$\,\psi^i\partial\nnh_i\w$, with 
$\,\psi^i\nh=u^j\partial\nh_j\w w^i\nh
+u^jw^k\vg_{\hskip-2.6ptjk}^{\hs i}$. Thus,
$\,\partial\nnh_a\w\psi^i\nh=0$, and $\,\ns u\w w\,$ is
$\,\vz\nh$-pro\-ject\-a\-ble.
\end{proof}
As noted in the above proof, in local coordinates chosen for
an integrable distribution $\,\vz\hs$ as in Lemma~\ref{lccoo}, the requirement
that $\,\vz\hs$ be $\,\nabla\nh$-par\-al\-lel amounts to
\begin{equation}\label{nbp}
\vg_{\hskip-2.6ptaj}^{\hs i}\,
=\,\hs\vg_{\hskip-2.6ptab}^{\hs i}\,=\,\hs0\hh.
\end{equation}
Pro\-ject\-a\-bil\-i\-ty of a tor\-sion-free connection $\,\nabla\hs$ on a 
manifold $\,\mf\,$ along a $\,\nabla\nh$-par\-al\-lel distribution $\,\vz\hs$
{\it is equivalent to the following condition imposed on the curvature tensor\/
$\,R\,$ of\/ $\,\nabla\nh$, at every point\/} $\,x\in\mf$:
\begin{equation}\label{crc}
\begin{array}{l}
R\hn_x\w(v,u)\hh u'\nh\in\hh\vx\,\mathrm{\ whenever\ }\,u,u'\in\txm\,\mathrm{\
and\ }\,v\in\vx.
\end{array}
\end{equation}
In fact, we may choose local coordinates for $\,\vz\hs$ as in
Lemma~\ref{lccoo}. Condition (\ref{crc}) reads
$\,R_{a\hs\bullet\bullet}\w{}^i\nh=0\,$ with
$\,\bullet\,$ standing for any index, that is, 
$\,R_{ajk}\w{}^i\nh=R_{ajb}\w{}^i\nh=R_{abj}\w{}^i\nh
=R_{abc}\w{}^i\nh=0$. If $\,\nabla\hs$ is
$\,\vz\nh$-pro\-ject\-a\-ble, Lemma~\ref{lccoo} gives (\ref{nbp}) and 
$\,\partial\nnh_a\w\vg_{\hskip-2.6ptjk}^{\hs i}\nh=0$,
which yields $\,R_{a\hs\bullet\bullet}\w{}^i\nh=0\,$ due to the
coordinate form of (\ref{cur}). Conversely, if (\ref{nbp}) holds and 
$\,R_{i\hs\bullet\bullet}\w{}^a\nh=0$, the coordinate form of (\ref{cur})
yields $\,\partial\nnh_a\w\vg_{\hskip-2.6ptjk}^{\hs i}\nh
=-\nh R_{ajk}\w{}^i\nh=\hs0$.

Given a bundle projection $\,\pi:M\to\bs\,$ with the vertical distribution
$\,\vz\nh=\text{\rm Ker}\hskip1.7ptd\pi\,$ and tor\-sion-free connections 
$\,\nabla\hs$ on $\,\mf\,$ and $\,D\,$ on $\,\bs$, the following three
conditions are mutually equivalent:
\begin{enumerate}
  \def\theenumi{{\rm\alph{enumi}}}
\item $\nabla\hs$ is $\,\vz\nh$-pro\-ject\-a\-ble onto $\,\bs\,$ with the
projected connection $\,D$,
\item the horizontal distribution of $\,\nabla\hs$ is
$\,d\pi$-pro\-ject\-a\-ble onto that of the $\,\pi$-pull\-back of $\,D$, the
vec\-tor-bun\-dle morphism $\,d\pi:T\nh\mf\nh\to\pi^*T\hn\bs\,$ being treated
as a mapping between the total spaces,
\item whenever $\,t\mapsto w(t)\in T\hskip-3.7pt_{x(t)}\w\hs\mf\,$ is a
$\,\nabla\nh$-par\-al\-lel vector field along a curve
$\,t\mapsto x(t)\in\mf\nh$, its $\,d\pi$-im\-a\-ge
$\,t\mapsto d\pi\nnh_{x(t)}\w w(t)\in T\hskip-2.8pt_{y(t)}\w\hh\bs\,$ is
$\,D$-par\-al\-lel along the image curve $\,t\mapsto y(t)=\pi(x(t))$.
\end{enumerate}
In fact, (b) and (c) imply each other: the left-to-right inclusion in
(\ref{dpv}) for $\,d\pi\,$ rather than $\,\pi\,$ trivially follows from (c). 
For the opposite inclusion, in suitable local coordinates,
$\,\pi:M\to\bs\,$ (and, consequently, $\,d\pi:T\nh\mf\nh\to\pi^*T\hn\bs$) 
appears as a linear projection $\,(y,\xi)\mapsto y\,$ (or, respectively,
$\,(y,\xi,\dot y,\dot\xi)\mapsto(y,\xi,\dot y)$), which realizes
any $\,D$-hor\-i\-zon\-tal vector as the image of a vector tangent to
$\,T\nh\mf\nh$, while a $\,T\nh\mf\nh$-ver\-ti\-cal correction allows us to
replace the latter with a $\,\nabla\nh$-hor\-i\-zon\-tal one.

To establish equivalence of (a) and (c), use local coordinates 
$\,x\hh^1\nh,\dots,x\hh^n$ in $\,\mf$ in which $\,\vz\,$ is spanned by the
last $\,\q\,$ coordinate vector fields, and so, with the index ranges
$\,i,j,k\in\{1,\dots,n-\q\}\,$ and $\,a,b\in\{n-\q+1,\dots,n\}$, we may treat
$\,x^{\hh i}$ as coordinates in $\,\bs$, denoting by
$\,\bvg_{\hskip-2.6ptjk}^{\hs i}$ the component functions of $\,D\,$ and
reserving the usual $\,\vg\hs$ notation for those of $\,\nabla\nh$. Condition
(c) now amounts to requiring that
$\,\dot w^{\hh i}\nh+\bvg_{\hskip-2.6ptjk}^{\hs i}\dot x^{\hh j}w^{\hh k}\nh
=0\,$ whenever
$\,\dot w^{\hh i}\nh+\vg_{\hskip-2.6ptjk}^{\hs i}\dot x^{\hh j}w^{\hh k}\nh
+\vg_{\hskip-2.6ptja}^{\hs i}\dot x^{\hh j}w^{\hh a}\nh
+\vg_{\hskip-2.6ptak}^{\hs i}\dot x^{\hh a}w^{\hh k}\nh
+\vg_{\hskip-2.6ptab}^{\hs i}\dot x^{\hh a}w^{\hh b}\nh=0$. Choosing
appropriate initial data we see that this amounts to 
$\,\vg_{\hskip-2.6ptaj}^{\hs i}\nh
=\nh\vg_{\hskip-2.6ptab}^{\hs i}\nh=0\,$ 
and $\,\vg_{\hskip-2.6ptjk}^{\hs i}\nh=\bvg_{\hskip-2.6ptjk}^{\hs i}\nh$. As
the last relation yields
$\,\partial\nnh_a\w\vg_{\hskip-2.6ptjk}^{\hs i}\nh=0$, our claim follows from 
Lemma~\ref{lccoo}.

\section{Null-pro\-ject\-a\-bil\-i\-ty of Le\-vi-Ci\-vi\-ta
connections}\label{np}
\setcounter{equation}{0}
For the Le\-vi-Ci\-vi\-ta connection $\,\nabla\hs$ of a
pseu\-\hbox{do\hs-}Riem\-ann\-i\-an metric on a manifold $\,\mf\,$ and 
an integrable distribution $\,\dz\hh$ on $\,\mf\nh$,
\begin{equation}\label{oco}
\dz\hyp\mathrm{pro\-ject\-a\-bil\-i\-ty\ of\ }\,\nabla\hs\mathrm{\ is\ 
equivalent\ to\ its\ }\,\dz\hh^\perp\nnh\hyp\mathrm{pro\-ject\-a\-bil\-i\-ty.}
\end{equation}
(Note that pro\-ject\-a\-bil\-i\-ty implies, by Lemma~\ref{prprl}(ii), that
the distribution in question is parallel; thus, $\,\dz\hh^\perp$ must be
integrable here if $\,\dz\hh$ is, and vice versa.) In fact, in terms of the 
$\,(0,4)\,$ curvature tensor, also denoted by $\,R$, (\ref{crc}) 
reads $\,R(\dz,\,\cdot\,,\,\cdot\,\dz\hh^\perp\nnh)=\{0\}$ which, due to
symmetries of $\,R$, amounts to 
$\,R(\dz\hh^\perp\nnh,\,\cdot\,,\,\cdot\,\dz)=\{0\}$, that is, to (\ref{crc})
with $\,\dz\hh$ replaced by $\,\dz\hh^\perp\nnh$.

This situation is of rather little interest when $\,\dz\hh$ is nondegenerate 
(meaning nondegeneracy of the metric restricted to $\,\dz$) since, the local
version of the de Rham decomposition theorem, originally due to
Thomas \cite{thomas}, then implies that $\,\dz$ is, locally, a factor
distribution in a product decomposition of the metric, and so $\,\nabla$
projects via (\ref{dpv}) onto the Le\-vi-Ci\-vi\-ta connection of the other
factor metric. The extreme opposite case, in which $\,\dz\hh$ is null, leads
to a much more diverse family of examples, such as those listed below.
\begin{example}\label{nlpar}Any nonzero null parallel vector field $\,w\,$ on
a pseu\-\hbox{do\hskip.7pt-}Riem\-ann\-i\-an manifold has
$\,R(w,\,\cdot\,,\,\cdot\,,\,\cdot)=0$, which implies (\ref{crc}) for the
distribution $\,\dz\hh$ spanned by $\,w$, and hence
$\,\dz\nh$-pro\-ject\-a\-bil\-i\-ty of the Le\-vi-Ci\-vi\-ta connection.
\end{example}
\begin{example}\label{ecsmf}{\it ECS manifolds\/} \cite{derdzinski-roter-padge}
are pseu\-\hbox{do\hskip.7pt-}Riem\-ann\-i\-an manifolds of dimensions
$\,n\ge4\,$ which have parallel Weyl tensor withous being con\-for\-mal\-ly
flat or locally symmetric. They exist for every $\,n\ge4\,$ as shown by Roter 
\cite[Corol\-lary~3]{roter}, their metrics are all indefinite 
\cite[Theorem~2]{derdzinski-roter-77}, and compact examples are known in all
dimensions $\,n\ge5\,$ \cite{derdzinski-roter-10,derdzinski-terek}. Every
ECS manifold carries a distinguished null parallel distribution $\,\dz\hh$ of
dimension $\,d\in\{1,2\}$, discovered by Ol\-szak \cite{olszak}, and its
Le\-vi-Ci\-vi\-ta connection is $\,\dz$-pro\-ject\-a\-ble: for $\,d=2$,
\cite[Lemma 17.3(ii)]{derdzinski-roter-tmj} yields (\ref{crc}) while, if
$\,d=1$, the metric has, locally, according to
\cite[Theorem 4.1]{derdzinski-roter-09}, the coordinate form of
\cite[formula (3.2)]{derdzinski-terek} with $\,\dz\hh$ spanned by a null
parallel vector field $\,w\,$
\cite[lines following formula (3.6)]{derdzinski-terek}, and we can invoke
Example~\ref{nlpar}.
\end{example}
\begin{example}\label{rmext}In a cotangent af\-fine bundle over a manifold 
carrying a tor\-sion-free connection, equipped with any Riemann extension
metric, the vertical distribution $\,\vz\hs$ is null and parallel, and the 
Le\-vi-Ci\-vi\-ta connection is $\,\vz\nh$-pro\-ject\-a\-ble. See
Section~\ref{re} below, especially Theorem~\ref{afifi}.
\end{example}
Let us call the Le\-vi-Ci\-vi\-ta connection of an $\,n$-di\-men\-sion\-al 
pseu\-\hbox{do\hs-}Riem\-ann\-i\-an manifold $\,(\mf\nh,g)\,$ {\it
null-pro\-ject\-a\-ble\/} if it is pro\-ject\-a\-ble along some null parallel
distribution $\,\dz\hh$ of dimension $\,r$, with $\,0<r<n$. We will use the
index ranges 
\begin{equation}\label{rin}
1\,\le\,i,j,k\,\le\,r\,<\,p,q\,\le\,n-r\,<\,a,b\,\le\,n\mathrm{,\ where\
}\,n=\dim\mf.
\end{equation}
As shown by Walk\-er \cite{walker}, a null parallel distribution $\,\dz\hh$ 
of dimension $\,r$ on an $\,n$-di\-men\-sion\-al 
pseu\-\hbox{do\hs-}Riem\-ann\-i\-an manifold is, in some local
coordinates, spanned by the last $\,r\hs$ coordinate vector fields, while,
for the components of the metric $\,g$,
\begin{equation}\label{gab}
\begin{array}{l}
g_{ab}\w=\,g_{ap}\w=\,0\hh,\hskip20pt\det[g_{ia}\w]\ne0\ne\det[g_{pq}\w]\hh,\\
\partial\nnh_a\w g_{ia}\w\nh=\hs\partial\nnh_p\w g_{ia}\w\nh
=\hs\partial\nnh_j\w g_{ia}\w\nh=\hs\partial\nnh_a\w g_{pq}\w\nh
=\hs\partial\nnh_a\w g_{pi}\w\nh=\,0\hh,
\end{array}
\end{equation}
$\det[g_{ia}\w]\ne0\ne\det[g_{pq}\w]\,$ reflecting nondegeneracy of $\,g$.
Conversely, (\ref{gab}) with (\ref{rin}) always defines a metric $\,g\,$ for 
which the span $\,\dz\hh$ of $\,\partial\nnh_{n-r+1}\w,\dots,\partial\nnh_n\w$
is null and parallel. Note that $\,[g_{ia}\w]\,$ is here a nonsingular
$\,r\times r\hs$
matrix of constants (and may always be assumed equal to the identity matrix).
\begin{lemma}\label{gmijk}For the Le\-vi-Ci\-vi\-ta connection of a metric\/
$\,g\,$ satisfying\/ {\rm(\ref{gab})} with\/ {\rm(\ref{rin})}, or even the 
weaker assumption that\/ 
$\,g_{ab}\w=g_{ap}\w=\partial\nnh_j\w g_{ia}\w=0$, one has\/
$\,2\vg_{\hskip-2.6ptjk}^{\hs i}\nh
=-g^{\hs ai}\partial\nnh_a\w g\nh_{jk}\w$, where the matrix\/
$\,[g^{\hs ai}]\,$ is the inverse of\/ $\,[g_{ia}\w]$.
\end{lemma}
In fact, $\,0=\delta_a^j=g\hh^{ji}g_{ia}\w$ and
$\,0=\delta_a^p=g\hh^{pi}g_{ia}\w$, so that $\,g\hh^{ji}\nh=g\hh^{pi}\nh=0\,$
since $\,\det[g_{ia}\w]\ne0$, and our claim follows.
\begin{lemma}\label{nbprj}The Le\-vi-Ci\-vi\-ta connection of a metric\/
$\,g\,$ as in\/ {\rm(\ref{gab})}  with\/ {\rm(\ref{rin})} is pro\-ject\-a\-ble
along the distribution\/
$\,\dz\hh$ spanned by the last\/ $\,r$ coordinate vector fields if and only
if\/ $\,\partial\nnh_a\w\partial\nh_b\w\hs g\nh_{jk}\w
=\partial\nnh_a\w\partial\nnh_p\w\hs g\nh_{jk}\w=0\,$ or, equivalently,
\begin{equation}\label{eqv}
g\nh_{jk}\w=x^{\hh a}\nnh B\nnh_{ajk}\w\nnh+\stf\hn_{jk}\w\mathrm{\ for\
some\ }\,B_{ajk}\w,\stf\hn_{jk}\w\mathrm{\ with\
}\,\partial\nh_b\w B\nnh_{ajk}\w=\partial\nnh_p\w B\nnh_{ajk}\w
=\partial\nnh_a\w \stf\hn_{jk}\w=0\hh.
\end{equation}
\end{lemma}
\begin{proof}According to Lemma~\ref{lccoo}, the requirement that 
$\,\nabla\hs$ be $\,\vz\nh$-pro\-ject\-a\-ble, where 
$\,\vz\nh=\dz\hh^\perp\nnh$,
equivalent -- by (\ref{oco}) -- to its $\,\dz$-pro\-ject\-a\-bil\-i\-ty,
reads, 
$\,\vg_{\hskip-2.6ptaj}^{\hs i}\nh
=\nh\vg_{\hskip-2.6ptab}^{\hs i}\nh
=\vg_{\hskip-2.6ptpj}^{\hs i}\nh
=\nh\vg_{\hskip-2.6ptap}^{\hs i}\nh
=\nh\vg_{\hskip-2.6ptpq}^{\hs i}\nh
=\nh\partial\nnh_a\w\vg_{\hskip-2.6ptjk}^{\hs i}\nh
=\nh\partial\nnh_p\w\vg_{\hskip-2.6ptjk}^{\hs i}\nh=0$. However,
$\,\vg_{\hskip-2.6ptaj}^{\hs i}\nh
=\nh\vg_{\hskip-2.6ptab}^{\hs i}\nh
=\vg_{\hskip-2.6ptpj}^{\hs i}\nh
=\nh\vg_{\hskip-2.6ptap}^{\hs i}\nh
=\nh\vg_{\hskip-2.6ptpq}^{\hs i}\nh=0$ amounts, in view of
(\ref{nbp}), to $\,\vz\hs$ being $\,\nabla\nh$-par\-al\-lel, which is the
case here, so that $\,\nabla$ is $\,\vz\nh$-pro\-ject\-a\-ble
(or $\,\dz$-pro\-ject\-a\-ble) if and only if
$\,\partial\nnh_a\w\vg_{\hskip-2.6ptjk}^{\hs i}
=\partial\nnh_p\w\vg_{\hskip-2.6ptjk}^{\hs i}=0$, and our claim follow
from Lemma~\ref{gmijk}.
\end{proof}
Lemma~\ref{nbprj} completely describes the local picture of
null-pro\-ject\-a\-bil\-i\-ty for Le\-vi-Ci\-vi\-ta connections, illustrating
the relative strength of the pro\-ject\-a\-bil\-i\-ty requirement
versus just assuming the distribution to be null and parallel.

In addition to being essentially trivial, Lemma~\ref{nbprj} also fails to
identify various geometric aspects of this situation which have a
co\-or\-di\-nate-free description. We address such aspects in Section~\ref{gc}.

\section{Cotangent af\-fine bundles}\label{ca}
\setcounter{equation}{0}
An {\it af\-fine bundle\/} over a manifold $\,\bs\,$ is defined in the usual
way, so that there is a total space $\,\mf\,$ with a bundle projection 
$\,\pi:M\to\bs$, each fibre $\,\mf\hskip-2.3pt_y\w=\pi^{-\nnh1}\nh(y)$ of
which carries a structure of an af\-fine space depending smoothly on
$\,y\in\bs$. It has its associated vector bundle $\,N\nh$, the fibre of which
over each $\,y\in\bs\,$ forms the translation vector space of the af\-fine
space $\,\mf\hskip-2.3pt_y\w$. Any fixed section $\,S\,$ of $\,\mf\nh$,
treated as a sub\-man\-i\-fold of the total space $\,\mf\nh$, allows us to
identify $\,\mf\,$ with $\,N\hs$ by providing in each fibre
$\,\mf\hskip-2.3pt_y\w$ the origin $\,o_y\w$ given by 
$\,S\cap\nh\mf\hskip-2.3pt_y\w=\{o_y\w\}$. Note that, the fibres being
contractible, a global section $\,S\,$ always exists.

We call such $\,\mf\,$ a {\it cotangent af\-fine bundle\/} over 
$\,\bs\,$ if $\,N\nh=\tab$.

Cotangent af\-fine bundles are encountered in a variety of 
interesting situations. One example arises in the case of a real line bundle
or a complex Her\-mit\-i\-an line bundle $\,\varLambda\,$ over $\,\bs$. The
linear connections (or, respectively, Her\-mit\-i\-an linear connections) in
$\,\varLambda\,$ then constitute precisely all the sections of an af\-fine
bundle over $\,\bs$ associated with $\,\tab$.

In the complex hol\-o\-mor\-phic category, with a fixed hol\-o\-mor\-phic
vector bundle $\,N\hs$ over a complex manifold $\,\bs$, the set of equivalence
classes of hol\-o\-mor\-phic af\-fine bundles over $\,\bs\,$ associated with
$\,N\hs$ stands in a natural one-to-one correspondence with the first
sheaf co\-ho\-mol\-o\-gy group $\,H^1\nh(\bs,\mathcal{F})$, for the sheaf
$\,\mathcal{F}\,$ of local hol\-o\-mor\-phic sections of $\,N\nh$. This
applies, in particular, to $\,N\nh=\tab$, the hol\-o\-mor\-phic cotangent
bundle of $\,\bs$.

We will introduce a generalization of cotangent af\-fine bundles in
Section~\ref{bc}. 

\section{Riemann extensions}\label{re}
\setcounter{equation}{0}
Given a cotangent af\-fine bundle $\,\mf\,$ over a manifold $\,\bs\,$
(Section~\ref{ca}), every \hbox{$1$-form} $\,\of\,$ on $\,\bs\,$ may be viewed
as a fibre-pre\-serv\-ing dif\-feo\-mor\-phism $\,\mf\nh\to\mf\nh$,
\begin{equation}\label{odf}
\mathrm{acting\ in\ each\ fibre\ }\,\mf\hskip-2.3pt_y\w\mathrm{\ via\ the\
translation\ by\ }\,\of_y\w\in\tayb.
\end{equation}
Let $\,\pi:\mf\nh\to\bs\,$ and $\,\vz=\,\mathrm{Ker}\hskip2.3ptd\pi$ denote the
bundle projection and the vertical distribution. By a {\it stand\-ard-type
metric\/} on $\,\mf\,$ we mean any pseu\-\hbox{do\hs-}Riem\-ann\-i\-an
metric $\,g\,$ on $\,\mf\,$ having $\,g_x\w(\xi,w)=\xi(d\pi\nh_x\w w)$ for
any $\,x\in\mf\nh$, any $\,w\in\txm\nh$, and any vertical vector
$\,\xi\in\vx=\tayb$, with $\,y=\pi(x)$. To define such $\,g$, it suffices to
fix a horizontal distribution $\,\hz\,$ on $\,\mf$, with
$\,T\nh\mf\nh=\vz\oplus\hn\hz$, and prescribe the restriction of $\,g\,$ to
$\,\hz\,$ (which may be any section of $\,[\hs\hz^*\nh]^{\odot2}\nh$, since
nondegeneracy of $\,g\,$ then follows as $\,g_x\w$ has a nonsingular matrix in
a ver\-ti\-cal-hor\-i\-zon\-tal basis of $\,\txm$). The vertical distribution
$\,\vz\hs$ being $\,g$-null, every stand\-ard-type metric $\,g\,$ has the
neutral metric signature.

Pat\-ter\-son and Walk\-er's {\it Riemann extensions\/}
\cite[p.\ 26]{patterson-walker} form a class of neutral
pseu\-\hbox{do\hs-}Riem\-ann\-i\-an metrics on cotangent af\-fine bundles
$\,\mf\,$ over any manifold $\,\bs$ equipped with a fixed tor\-sion-free
connection $\,D$. We define them here to be those 
stand\-ard-type metrics on $\,\mf\,$ which, for every $\,1$-form $\,\of\,$ on
$\,\bs\,$ treated as a dif\-feo\-mor\-phism $\,\of:\mf\nh\to\mf\,$ with
(\ref{odf}), satisfy \cite[\S8]{patterson-walker} the transformation rule
\begin{equation}\label{rem}
\of\hh^*\nnh g\,\,=\,\,g\,\,+\,\hs\pi^*\nnh\ko\of\hh,
\end{equation}
$\ko\,$ being the {\it Kil\-ling operator\/} associated with $\,D$, 
sending any $\,1$-form $\,\of\,$ on $\,\bs$ to the symmetric 
twice-co\-var\-i\-ant tensor field
\begin{equation}\label{kil}
\ko\of\,=\,D\hh\of\,+\,[D\hh\of]^*\hskip5pt\mathrm{or,\ in\ coordinates,\ \
}\,[\ko\of]_{ij}\w=\,\,\of_{i,j}\w+\,\,\of_{j,i}\w\hh.
\end{equation}
Thus, $\,\of\,$ is a $\,g$-isom\-e\-try if $\,\ko\of=0$.

We use the term `Riemann extension\nh' narrowly.  Wider classes of Riemann 
extensions have been discussed in the literature. See, for instance, 
\cite{patterson-walker}, \cite{afifi} and, more recently, 
\cite{calvino-louzao-garcia-rio-gilkey-vazquez-lorenzo}. 
As the next lemma shows, the above definition of Riemann extensions is
equivalent to the standard one, appearing in
\cite[formula (28)]{patterson-walker}.
\begin{lemma}\label{eqdef}Let\/ $\,\mf\,$ be the total space of a cotangent
af\-fine bundle over an\/ $\,r\hn$-di\-men\-sion\-al manifold\/ $\,\bs\,$
carrying a tor\-sion-free connection\/ $\,D$.
Given any sub\-man\-i\-fold\/ $\,S\,$ of\/ $\,\mf\,$ forming a
global section of the af\-fine bundle, and any symmetric twice-co\-var\-i\-ant
tensor field\/ $\,\stf\,$ on\/ $\,S$, there exists a unique Riemann extension
metric\/ $\,g\,$ for\/ $\,D\,$ the restriction of which to\/ $\,S\,$
equals\/ $\,\stf$.

If\/ $\,S\,$ is used to identify\/ $\,\mf\,$ with $\,\tab$, so as to turn\/
$\,S\,$ into the zero section\/ $\,\bs\subseteq\tab$, then, in local
coordinates\/ $\,x^{\hh i}\nnh,\xi_{\hh i}\w\hskip2pt$ for\/ $\,\tab\,$
arising from a coordinate system\/ $\,x^{\hh i}$ for\/ $\,\bs\,$ in which\/
$\,D\,$ has the components\/ $\,\vg_{\hskip-2.6ptjk}^{\hs i}$,
\begin{equation}\label{rex}
g\,=\,2\hs d\hh\xi_{\hh i}\w\nnh\odot\hh dx^{\hh i}\hs+\,(\stf\hn_{jk}\w\hs
-\,2\hs\xi_{\hh i}\w\vg_{\hskip-2.6ptjk}^{\hs i})\,
dx^{\hh j}\nnh\odot\hh dx^{\hh k}\nnh.
\end{equation}
With the index ranges\/ $\,i,j,k\in\{1,\dots,r\}\,$ and\/
$\,a,b\in\{r+1,\dots,2r\}$, 
using any fixed 
nonsingular\/ $\,r\times r\hs$ matrix\/ $\,[g_{ia}\w]\,$ of constants and
its inverse\/ $\,[g^{\hs ai}]$, we may set\/
$\,x^{\hh a}\nh=g^{ai}\xi_{\hh i}\w$. In the resulting coordinates\/
$\,x^{\hh1}\nh,\dots,x^{\hh2r}\nh$, {\rm(\ref{rex})} reads\/
\hbox{$\,g=2g_{ia}\w dx^{\hh i}\nnh\odot\hh dx^{\hh a}$}
\hbox{$+(\stf\hn_{jk}\w
-2g_{ia}\w x^{\hh a}\vg_{\hskip-2.6ptjk}^{\hs i})\,
dx^{\hh j}\nnh\odot\hh dx^{\hh k}$,} that is, $\,g\,$ has the components\/
\begin{equation}\label{gcp}
g\nh_{jk}\w=\stf\hn_{jk}\w
-2g_{ia}\w x^{\hh a}\nnh\vg_{\hskip-2.6ptjk}^{\hs i}\,\mathrm{\ and\
}\,g_{ab}\w=0\mathrm{,\ along\ with\ our\ fixed\ constants\ }\,g_{ia}\w\hh.
\end{equation}
\end{lemma}
\begin{proof}The existence claim is immediate since (\ref{rex}) defines a
metric $\,g\,$ with the required properties: the zero section
$\,\bs\subseteq\tab\,$
being given by $\,\xi_{\hh i}\w=0$, the restriction of $\,g\,$ to $\,\bs\,$ 
is nothing else than $\,\stf$, while (\ref{rem}) follows as the
$\,\of$-pull\-back operation applied to differential forms on $\,\mf\,$
commutes with $\,d$, leaves invariant functions on $\,\bs\,$ treated as defined
on $\,\mf\nh$, and
$\,\of\hh^*\hn\xi_i\w=\xi_i\w\circ\of=\xi_i\w+\of_{\hn i}\w$,
so that $\,\of\hh^*\nnh g$ equals the right-hand side of (\ref{rex}) plus
$\,(\partial\nh\nnh_j\w\of_{\nh i}\w+\partial\nnh_i\w\of_{\nh j}\w
-2\hs\of_{\nh k}\w\vg_{\hskip-2.6ptij}^{\hs k})
\,dx^{\hh i}\nnh\odot\hh dx^{\hh j}\nh=\ko\of$.

Uniqueness of $\,g\,$ follows easily as well: given $\,x\in\mf\nh$, let
$\,y=\pi(x)\in\bs$, so that $\,x\,$ lies in the affine space
$\,\my=\pi^{-\nnh1}\nh(y)$. If we fix a $\,1$-form $\,\of\,$ on
$\,\bs\,$ for which $\,o=x+\of\nh_y\w$ is the unique intersection point of 
$\,S\,$ and $\,\my$ (the origin in $\,\my$ provided by the
global section $\,S$) then, by (\ref{rem}),
$\,g_x\w=\of_{\nnh x}^*g_o\w-[\pi^*\nnh\ko\of]\hn_x\w$, which proves our claim as
$\,x\,$ and $\,S\,$ uniquely determine $\,o$.
\end{proof}
The following intrinsic local characterization of Riemann extension metrics is
a special case of a result of Afifi \cite[p.\ 313]{afifi}. See also
\cite[p.\ 369,\ Theorem 4.5]{derdzinski}. We provide a proof here for the
reader's convenience.
\begin{theorem}\label{afifi}Given a tor\-sion-free connection\/ $\,D$ on
a manifold\/ $\,\bs$ and a Riemann extension 
metric\/ $\,g\,$ for\/ $\,D$ on the total space\/ $\,\mf$ of a cotangent 
af\-fine bundle over\/ $\,\bs$, the Le\-vi-Ci\-vi\-ta connection\/ $\,\nabla$ 
of\/ $\,g\,$ is pro\-ject\-a\-ble along the vertical distribution\/ $\,\vz$ ot
the bundle projection\/ $\,\mf\nh\to\bs$, while $\,\vz$ itself is\/ $\,g$-null 
as well as\/ $\,\nabla\nh$-par\-al\-lel, and the projected tor\-sion-free
connection described in Lemma\/~{\rm\ref{prjcn}} coincides with\/ $\,D$.

Conversely, let the Le\-vi-Ci\-vi\-ta connection\/ $\,\nabla$ of a
pseu\-d\hbox{o\hs-}\hskip0ptRiem\-ann\-i\-an manifold\/ $\,(\mf\nh,g)\,$
with\/ $\,\dim\mf\nh=2r$ be pro\-ject\-a\-ble along an 
\hbox{$\,r$\hh-}\hskip0ptdi\-men\-sion\-al null parallel distribution\/ 
$\,\vz\nh$. Then for  every point\/ $\,x\in\mf\nh$, there exist a 
manifold\/ $\,\bs\,$ of dimension\/ $\,r$, a tor\-sion\-free connection\/ 
$\,D$ on\/ $\,\bs$, and a dif\-feo\-mor\-phic identification of a
neighborhood\/ of\/ $\,x\,$ in\/ $\,\mf\hs$ with an open subset of the total
space of a cotangent af\-fine bundle over\/ $\,\bs$, under which\/ 
$\,g,\vz\hs$ and the projected tor\-sion-free connection on a local 
leaf space correspond to a Riemann extension metric for\/ 
$\,D$, the vertical distribution of the af\-fine-bun\-dle projection, and\/
$\,D$.
\end{theorem}
\begin{proof}The first part is immediate from Lemmas~\ref{nbprj}
and~\ref{lccoo}: (\ref{gcp}) amounts to (\ref{gab}), where
the indices $\,p,q\,$ have an empty range, with (\ref{eqv}) for
$\,B\nnh_{ajk}\w=-2g_{ia}\w\nh\vg_{\hskip-2.6ptjk}^{\hs i}$ and our 
$\,\stf\hn_{jk}\w$. 
For the second part, Walk\-er's theorem \cite{walker} 
applied to $\,\dz=\vz\hs$ gives (\ref{gab}) with an empty range for the
indices $\,p,q$, while (\ref{eqv}) now becomes (\ref{gcp}) if one defines
$\,\vg_{\hskip-2.6ptjk}^{\hs i}\,$ by 
$\,B\nnh_{ajk}\w=-2g_{ia}\w\nh\vg_{\hskip-2.6ptjk}^{\hs i}$ using any fixed 
nonsingular $\,r\times r\hs$ matrix $\,[g_{ia}\w]\,$ of constants.
Lemma~\ref{nbprj} and the final clause of Lemma~\ref{lccoo} then yield our
claim.
\end{proof}
The transformation rule (\ref{rem}) involves two actions by the
in\-fi\-nite-di\-men\-sion\-al Abel\-i\-an group
$\,\varOmega^1\hskip-2pt\bs\,$ of all 
$\,1$-forms $\,\of\,$ on $\,\bs$, one via the $\,\of$-pull\-back, the
other -- via the addition of $\,\pi^*\nh\ko\of$. The two actions of 
$\,\varOmega^1\hskip-2pt\bs\,$ commute with each other, which allows us to 
think of {\it Riemann extension metrics on\/} $\,\mf\,$ {\it as the fixed 
points of a specific action of\/} $\,\varOmega^1\hskip-2pt\bs\,$ on the the set
of stand\-ard-type metrics on $\,\mf\,$ (which both actions leave invariant).

\section{Geometric consequences of null-pro\-ject\-a\-bil\-i\-ty}\label{gc}
\setcounter{equation}{0}
Let $\,\dz\,$ be an $\,r$-di\-men\-sion\-al null parallel distribution on a 
pseu\-d\hbox{o\hs-}\hskip0ptRiem\-ann\-i\-an manifold $\,(\mf\nh,g)\,$
with $\,\dim\mf\nh=n$. Every point of $\,\mf\,$ has a neighborhood $\,\,U$
such that, for some manifolds $\,\bs\,$ and $\,\qe\,$ of dimensions $\,r\hh$
and $\,n-2r$, there are three bundle projections, and two vertical
distributions of interest to us:
\begin{equation}\label{bpv}
\mq:U\nh\to\qe\hh,\hskip5.3pt
\qs:\qe\to\bs\hh,\hskip5.3pt\pi=\qs\circ\mq:U\nh\to\bs\hh,\hskip5.3pt
\dz\hh^\perp\nh=\text{\rm Ker}\hskip1.7ptd\pi\hh,\hskip5.3pt
\dz=\text{\rm Ker}\hskip1.7ptd\hh\mq\hh.,
\end{equation}
as one sees applying (\ref{fbr}) twice, first to realize 
$\,\dz$, locally, as the vertical 
distribution of a bundle projection and, noting that
$\,\vz\nh=\dz\hh^\perp\nh$ then projects 
onto an integrable distribution on the base manifold, to similarly realize
this projected distribution.

For the remainder of this section, assume that the Le\-vi-Ci\-vi\-ta
connection $\,\nabla$ \hbox{of $\,(\mf\nh,g)\,$} is pro\-ject\-a\-ble along
$\,\dz\,$ or, equivalently due to (\ref{oco}), along
$\,\vz\nh=\dz\hh^\perp\nh$.

With $\,\pi\,$ and $\,\mq\,$ appearing in (\ref{bpv}), and any $\,y\in\bs$,
the assignment
\begin{equation}\label{pxg}
\tayb\ni\xi\mapsto v\,\,\mathrm{\ such\ that\ }\,\,\pi^*\nh\xi\,
=\,g(v,\,\cdot\,)\hh,
\end{equation}
is a linear iso\-mor\-phism between $\,\tayb\,$ and the space of
all vector fields $\,v\,$ tangent to $\,\dz\nh$, defined just on the leaf
$\,\pi^{-\nnh1}\nh(y)\,$ of $\,\vz\nh=\mathcal{P}^\perp$ and parallel along
this leaf. That $\,v$ must be tangent to $\,\dz\,$ follows as
$\,\,\pi^*\nh\xi\,$ vanishes on $\,\vz\nh=\mathcal{P}^\perp\nh$. To show that
$\,\ns w\w v=0$ for any vector field $\,w\,$ on $\,\,U\,$ tangent to 
$\,\vz\nh$, we are free to assume that the $\,1$-form $\,\xi\,$ is defined on
$\,\bs\,$ rather that just at $\,y$. Pro\-ject\-a\-bil\-i\-ty of $\,\nabla\nh$
along $\,\vz\nh$, with the projected connection $\,D\,$ on $\,\bs$, easily
gives $\,g(\ns w\w v,\,\cdot\,)=\ns w\w[\pi^*\nh\xi]=\pi^*\nh[D\nh_w\w\xi]\,$
for {\it all\/} $\,\pi$-pro\-ject\-a\-ble vector fields, where $\,w\,$ also
denotes the projected image of $\,w$. In our case, $\,w\,$ projects onto
$\,0$, and our claim follows.

On the other hand, for each fibre $\,\qe\nnh_y\w=\qs^{-\nnh1}(y)\,$ of the
bundle $\,\qe$, where $\,y\in\bs$,
\begin{equation}\label{vrm}
\begin{array}{l}
\mathrm{there\ exists\ a\ unique\ pseu\-do}\hyp\mathrm{Riem\-ann\-i\-an\ 
metric\ }\,\vh\nh_y\w\mathrm{\ on\ }\,\qe\nnh_y\w\mathrm{\ such\
that}\\
\mq^*\nh\vh\nh_y\w\mathrm{\ equals\ the\ restriction\ of\
}\,g\,\mathrm{\ to\ the\ leaf\ 
}\,\,U\hskip-3pt_y\w=\hs\pi^{-\nnh1}\nh(y)\,\hs\mathrm{\ of\
}\,\hs\vz=\hs\dz^\perp\nnh. 
\end{array}
\end{equation}
In fact, the restriction of $\,g\,$ to
$\,\,U\hskip-3pt_y\w$ is pro\-ject\-a\-ble under 
$\,\mq:U\hskip-3pt_y\w\to\qe\nnh_y\w$. Namely, if 
$\,\mq$-pro\-ject\-a\-ble vector fields $\,v,w\,$ on $\,\,U\,$ are tangent to
$\,\vz\nh$, and a vector field $\,u\,$ tangent to $\,\dz\nh$, we have
$\,d_u\w[g(v,w)]=g(\ns u\w v,w)+g(v,\ns u\w w)
=g(\ns v\w u,w)+g(v,\ns w\w u)=0$ \hbox{since 
$\,[u,v]\,$} and $\,[u,w]\,$ are tangent to $\,\dz\,$ by (\ref{prj}), and so
are $\,\ns v\w u,\ns w\w u\,$
as $\,\dz\,$ is parallel, while $\,\vz\nh=\dz^\perp$ (which also proves 
nondegeneracy of the projected metric).

\section{Bundles of co\-tan\-gent-prin\-ci\-pal bundles}\label{bc}
\setcounter{equation}{0}
We call a manifold $\,\mf\,$ a {\it bundle of co\-tan\-gent-prin\-ci\-pal
bundles\/} over a base manifold $\,\bs\,$ if we are given a bundle projection
$\,\qs:\qe\to\bs\,$ with some total space $\,\qe$, while $\,\mf\,$ itself is
the total space of an af\-fine bundle over $\,\qe\,$ having the associated
vector bundle $\,\qs^*\tab\,$ (the $\,\qs$-pull\-back of $\,\tab$). This leads
to three bundle projections and two relevant vertical distributions:
\begin{equation}\label{tbp}
\mq:\mf\nh\to\qe\hh,\hskip5.3pt
\qs:\qe\to\bs\hh,\hskip5.3pt\pi=\qs\circ\mq:\mf\nh\to\bs\hh,\hskip5.3pt
\vz\nh=\text{\rm Ker}\hskip1.7ptd\pi\hh,\hskip5.3pt
\dz=\text{\rm Ker}\hskip1.7ptd\hh\mq\hh.
\end{equation}
The projection $\,\pi:\mf\nh\to\bs\,$ turns $\,\mf\,$ into a bundle over
$\,\bs\,$ with the fibre $\,\my$ 
over each $\,y\in\bs$ arising as the restriction of the af\-fine bundle
$\,\mf\,$ over $\,\qe\,$ to the fibre
$\,\qe\nnh_y\w=\qs^{-\nnh1}(y)\subseteq\qe$, and
hence forming a $\,\tayb$-prin\-ci\-pal bundle over $\,\qe\nnh_y\w$ (in the
usual sense, for the additive group $\,\tayb$).

Every section $\,\of\,$ of the vector bundle $\,\qs^*\tab\,$ over $\,\qe\,$ is
now a fibre-pre\-serv\-ing
\begin{equation}\label{dif}
\mathrm{dif\-feo\-mor\-phism\ }\,\of:\mf\nh\to\mf\nh\mathrm{,\ acting\ as\ the\
translation\ by\ }\,\of\nh_z\w\in\tayb
\end{equation}
in each fibre $\,\mq^{-\nnh1}(z)\,$ over $\,z\in\qe$, the fibre being an
af\-fine space with the translation vector space $\,\tayb$, where $\,y=\qs(z)$.

We are particularly interested in the case where, for a given bundle of
co\-tan\-gent-prin\-ci\-pal bundles, with $\,\mf\nh,\qe,\bs\,$ as above
and (\ref{tbp}),
\begin{enumerate}
  \def\theenumi{{\rm\alph{enumi}}}
\item[(a)] $\bs\,$ carries a tor\-sion-free connection $\,D$.
\end{enumerate}
Then the Kil\-ling operator $\,\ko\,$ given by (\ref{kil}) can be
extended to act on sections $\,\of$ of $\,\qs^*\tab$. Namely, in
$\,\qs^*\tab\,$ one has the $\,\qs$-pull\-back of the connection in $\,\tab$
dual to $\,D$. For simplicity, the dual connection and its $\,\qs$-pull\-back
are still denoted here by $\,D$. Thus, $\,[D\hh\of]_z\w$, at every $\,z\in\qe$,
is a linear operator $\,\tzq\to\tayb$, with $\,y=\qs(z)$, assigning to
$\,v\in\tzq\,$ the $\,1$-form $\,\beta=D\nnh_v\w\hh\of\in\tayb$, and hence
giving rise to the bi\-lin\-e\-ar form
$\,\tzq\times\tzq\ni(v,u)\mapsto[D\hh\of]_z\w(v,u)=\beta(d\hh\qs_z\w u)$. In
this way $\,D\hh\of$ is interpreted as a twice-co\-var\-i\-ant tensor field on
$\,\qe$, and the Kil\-ling operator $\,\ko$ sends $\,\of\hh$ to twice the
symmetrization of $\,D\hh\of$, that is, again,
\begin{equation}\label{kli}
\ko\of\,=\,D\hh\of\,+\,[D\hh\of]^*\nh.
\end{equation}
The other assumption to be made about a given bundle of
co\-tan\-gent-prin\-ci\-pal bundles, for $\,\mf\nh,\qe,\bs\,$ as above, reads
\begin{enumerate}
  \def\theenumi{{\rm\alph{enumi}}}
\item[(b)] $\qe\,$ is a bundle of pseu\-\hbox{do\hs-}Riem\-ann\-i\-an manifolds
over $\,\bs$.
\end{enumerate}
Equivalently, (b) states that we have an assignment
\begin{equation}\label{yhy}
\bs\,\ni\,y\,\mapsto\,\vh\nh_y\w
\end{equation}
of a pseu\-\hbox{do\hs-}Riem\-ann\-i\-an metric $\,\vh\nh_y\w$ on 
$\,\qe\nnh_y\w=\qs^{-\nnh1}(y)\,$ to each $\,y\in\bs$, depending smoothly on 
$\,y$. In other words, $\,\qe\,$ is endowed with a {\it vertical metric\/}
$\,\vh$, meaning a
pseu\-\hbox{do\hs-}Riem\-ann\-i\-an fibre metric in the vertical distribution
$\,\text{\rm Ker}\hskip1.7ptd\hh\qs$.

Bundles of co\-tan\-gent-prin\-ci\-pal bundles include cotangent af\-fine
bundles (Section~\ref{ca}) as a special case, 
with $\,\qe=\bs\,$ and $\,\qs=\mathrm{Id}$, so that $\,\qe\,$ has the 
sin\-gle-point fibres $\,\qe\nnh_y\w=\{y\}$, and the
\hbox{$\tayb$-prin\-ci\-pal} bundle over each $\,\{y\}\,$ is a single af\-fine
space associated with $\,\tayb$.

\section{Riemann pull\-back-ex\-ten\-sions}\label{rp}
\setcounter{equation}{0}
Let $\,\mf\,$ be a bundle of co\-tan\-gent-prin\-ci\-pal bundles over
$\,\bs\,$ (Section~\ref{bc}). In analogy with Section~\ref{re}, we define a
{\it stand\-ard-type metric\/} on $\,\mf\,$ to be any
pseu\-\hbox{do\hs-}Riem\-ann\-i\-an metric $\,g\,$ on $\,\mf\,$ such that,
with the notation of (\ref{tbp}), $\,g_x\w(\xi,w)=\xi(d\pi\nh_x\w w)\,$ for
any $\,x\in\mf\nh$, any $\,w\in\txm\nh$, and any vertical vector
$\,\xi\in\dzx=\tayb$, where $\,y=\pi(x)\in\bs\,$ (and $\,\dzx=\tayb\,$ since
$\,\dzx$ is the tangent space at $\,x\,$ of the af\-fine space
$\,\mq^{-\nnh1}(\mq\hh(x))\,$ with the translation vector space $\,\tayb$).

To define such $\,g$, it suffices, again, to fix a horizontal distribution
$\,\hz\,$ in the af\-fine bundle $\,\mf\,$ over $\,\qe$, with
$\,T\nh\mf\nh=\dz\oplus\hz$, and prescribe the restriction of $\hs g\hs$ to 
$\,\hz$. The restriction may, this time, be any section of
$\,[\hs\hz^*\nh]^{\odot2}$ nondegenerate on 
the sub\-bun\-dle $\,\hz^{\hs\mathrm{o}}$ of $\,\hz\,$ with the fibre
$\,\hz_{\hskip-1ptx}^{\hs\mathrm{o}}$ at each $\,x\in\mf\,$ equal to
the pre\-im\-age under the isomorphism
$\,d\hh\mq_x\w:\hz\nnh_x\w\to \tzq\,$ of the sub\-space 
$\,\tzq\nnh_y\w\subseteq \tzq$, for
$\,z=\mq\hh(x)\in\qe$ and $\,\qe\nnh_y\w=\qs^{-\nnh1}(y)$, where 
$\,y=\pi(x)\in\bs$. Nondegeneracy of $\,g\,$ then follows since, in a basis of
$\,\txm$ containing bases of $\,\dzx$ and
$\,\hz_{\hskip-1ptx}^{\hs\mathrm{o}}$, the matrix of $\,g_x\w$ has the form
\[
\left[\begin{matrix}0&0&G\cr
0&H&*\cr
G'&*&*\end{matrix}\right]
\]
for some nonsingular square matrices $\,G,H,G'\nh$.

Suppose that a bundle $\,\mf\,$ of co\-tan\-gent-prin\-ci\-pal bundles over
$\,\bs$, with the three bundle projections in (\ref{tbp}), also satisfies
(a) -- (b) in Section~\ref{bc}: one has a fixed tor\-sion-free connection
$\,D\,$ on $\,\bs\,$ and a vertical metric $\,\vh\,$ on $\,\qe$.

By a {\it Riemann pull\-back-ex\-ten\-sion\/} for $\,D,\qe\,$ and $\,\vh\,$ we
then mean any stand\-ard-type metric $\,g\,$ on $\,M\,$ satisfying the
transformation rule (\ref{rem}) for all sections $\,\of$ of the bundle
$\,\qs^*\tab\,$ over $\,\qe\,$ and having the property that the restriction of 
$\,g\,$ to each fibre $\,\my=\pi^{-\nnh1}\nh(y)\,$ of the bundle
$\,\mf\,$ over $\,\bs\,$ equals $\,\mq^*\nh\vh\nh_y\w$. Both requirements make
sense: the former due to (\ref{dif}) -- (\ref{kli}), the latter since $\,\mq\,$
maps $\,\my$ into $\,\qe\nnh_y\w$, and $\,\vh\nh_y\w$ is a metric on
$\,\qe\nnh_y\w$.

For a global section $\,S\,$ of the  bundle $\,\mf\,$ over $\,\qe$, treated as
a sub\-man\-i\-fold of $\,\mf\nh$, and a symmetric twice-co\-var\-i\-ant tensor
field $\,\stf\,$ on $\,S$, we call $\,\stf\,$ {\it consistent with our 
vertical metric\/} $\,\vh\,$ if the push-for\-ward of $\,\stf\,$ under the
dif\-feo\-mor\-phism $\,\mq:S\to\qe$, when restricted to the
$\,\qs$-ver\-ti\-cal sub\-bun\-dle of\/ $\,T\hn\qe$, equals $\,\vh$.

The next result generalizes Lemma~\ref{eqdef}. Note that the af\-fine bundle
$\,\mf\,$ over $\,\qe\,$ admits global sections.
\begin{lemma}\label{gnqdf}Let there be given a bundle of
co\-tan\-gent-prin\-ci\-pal bundles, with\/ $\,\mf\nh,\qe,\bs,D,\vh\,$ as in
Section\/~{\rm\ref{bc}}, {\rm(\ref{tbp})} -- {\rm(\ref{yhy})} and\/ {\rm(a)}
-- {\rm(b)}, Then, for any sub\-man\-i\-fold\/ $\,S\,$ of\/ $\,\mf\,$ forming
a global section of the  bundle\/ $\,\mf\,$ over\/ $\,\qe$, and any symmetric
twice-co\-var\-i\-ant tensor field\/ $\,\stf\,$ on\/ $\,S\,$ consistent with\/
$\,\vh$, there exists a unique Riemann pull\-back-ex\-ten\-sion metric\/
$\,g\,$ for\/ $\,D,\qe\,$ and\/ $\,\vh\,$ with the restriction to\/ $\,S\,$
equal to\/ $\,\stf$.

Using the index ranges\/ {\rm(\ref{rin})} and 
any fixed nonsingular\/ $\,r\times r\hs$ matrix\/ $\,[g_{ia}\w]\,$ of 
constants, in suitable local coordinates\/ $\,x^{\hh1}\nh,\dots,x^{\hh n}\nh$, 
we can express this unique\/ $\,g\,$ as
\begin{equation}\label{get}
g=2g_{ia}\w dx^{\hh i}\nnh\odot\hh dx^{\hh a}\nh
+(\stf\hn_{jk}\w\nh-2g_{ia}\w x^{\hh a}\nnh\vg_{\hskip-2.6ptjk}^{\hs i})
\,dx^{\hh j}\nnh\odot\hh dx^{\hh k}\nnh
+(2\stf\hn_{iq}\w\hs dx^{\hh i}\nh+\vh_{pq}\w\hs dx^{\hh p})\hs
\odot\nh dx^{\hh q}\nh,
\end{equation}
with functions\/
$\,\stf\hn_{jk}\w,\,\stf\hn_{iq}\w,\vh_{pq}\w,
\,\vg_{\hskip-2.6ptjk}^{\hs i}$ satisfying the conditions
\begin{equation}\label{cdt}
\partial\nnh_a\w\stf\hn_{jk}\w\hs=\,\partial\nnh_a\w\stf\hn_{ip}\w\hs
=\,\partial\nnh_a\w\vh_{pq}\w\hs
=\,\partial\nnh_p\w\vg_{\hskip-2.6ptjk}^{\hs i}\hs
=\,\partial\nnh_a\w\vg_{\hskip-2.6ptjk}^{\hs i}\hs=\,0\hh,\hskip7pt
|\mathrm{det}\hs[g_{pq}\w]|>0\hh.
\end{equation}
Equivalently, in addition to our fixed constants\/ $\,g_{ia}\w$, the
components of\/ $\,g\,$ are
\begin{equation}\label{ghc}
g\nh_{jk}\w\nh=\nh\stf\hn_{jk}\w\nh
-\nh2g_{ia}\w x^{\hh a}\nnh\vg_{\hskip-2.6ptjk}^{\hs i},
\,\,g_{ip}\w\nh=\stf\hn_{ip}\w,
\,\,g_{pq}\w\nnh=\hn\vh_{pq}\w,\,\,g_{pa}\w\nh=g_{ab}\w\nh=0.
\end{equation}
The coordinates\/ $\,x^{\hh i}\nh,x^{\hh p}\nh,x^{\hh a}$ to\/ may be chosen
so that\/ $\,x^{\hh i}$ and\/ $\,\xi_{\hh i}\w=g_{ia}\w x^{\hh a}\nh$ are 
local coordinates for\/ $\,\tab\,$ arising from the coordinate system\/
$\,x^{\hh i}$ for\/ $\,\bs\,$ in which\/ $\hs D\hs$ has the components\/ 
$\,\vg_{\hskip-2.6ptjk}^{\hs i}$, while\/ $\,x^{\hh i}\nh,x^{\hh p}$ 
form a coordinate system in\/ $\,\qe\,$ in which\/ $\,\stf$ and\/
$\,\vh\,$ have the components\/ $\,\stf\hn_{ij}\w$,
$\,\stf\hn_{iq}\w=\stf\hn_{qi}\w$ and\/ $\,\stf\hn_{pq}\w=\vh_{pq}\w$,  
and the first two projections in\/ 
{\rm(\ref{tbp})} send\/ $\,x^{\hh i}\nh,x^{\hh p}\nh,x^{\hh a}$ to\/ 
$\,x^{\hh i}\nh,x^{\hh p}$ and, respectively, $\,x^{\hh i}\nh,x^{\hh p}$ to\/ 
$\,x^{\hh i}\nh$. Also,
\begin{equation}\label{hnp}
g=2\hs d\hh\xi_{\hh i}\w\nnh\odot\hh dx^{\hh i}\nh
+(\stf\hn_{jk}\w\nh-2\xi_{\hh i}\w\vg_{\hskip-2.6ptjk}^{\hs i})
\,dx^{\hh j}\nnh\odot\hh dx^{\hh k}\nnh
+(2\stf\hn_{iq}\w\hs dx^{\hh i}\nh+\vh_{pq}\w\hs dx^{\hh p})\hn
\odot\nh dx^{\hh q}\nh,
\end{equation}
\end{lemma}
\begin{proof}Existence: with $\,\xi_{\hh i}\w=g_{ia}\w x^{\hh a}\nh$,
(\ref{get}) rewritten as (\ref{hnp}) defines a metric $\,g\,$ with the
required properties. Namely, $\,S$, identified 
with the zero section $\,\qe\,$ of $\,\qs^*\tab\,$ is given by
$\,\xi_{\hh i}\w=0$, and so $\,g\,$ restricted to $\,\qe\,$ equals $\,\stf\,$
(as $\,\stf\hn_{pq}\w=\vh_{pq}\w$ due to the assumption about consistency).
Finally, (\ref{rem}) follows since the
$\,\of$-pull\-back operation applied to differential forms on $\,\mf\,$
commutes with $\,d$, preserves functions on $\,\qe\,$ viewed as defined 
on $\,\mf\nh$, and $\,\of^*\nh\xi_i\w=\xi_i\w\circ\of=\xi_i\w+\of_{\hn i}\w$,
so that $\,\of^*\nh g\,$ equals the right-hand side of (\ref{hnp}) plus
$\,(\partial\nh\nnh_j\w\of_{\nh i}\w+\partial\nnh_i\w\of_{\nh j}\w
-2\hs\of_{\nh k}\w\vg_{\hskip-2.6ptij}^{\hs k})
\,dx^{\hh i}\nnh\odot\hh dx^{\hh j}\nh
+2\hh(\nh\partial\nh\nnh_p\w\of_{\nh i}\w\nnh)\hs
dx^{\hh i}\nnh\odot\hh dx^{\hh p}\nh=\ko\of$.

Uniqueness of $\,g\,$ follows easily as well: given $\,x\in\mf\nh$, let
$\,z=\mq\hh(x)\in\qe\,$ and $\,y=\qs(z)=\pi(x)\in\bs$, so that $\,x\,$ lies in 
the affine space $\,\mq^{-\nnh1}\nh(z)$. If we fix a $\,1$-form $\,\of\,$
on $\,\bs\,$ for which $\,o=x+\of\nh_y\w$ is the unique intersection point of 
$\,S\,$ and $\,\mq^{-\nnh1}\nh(z)$ (the origin in $\,\mq^{-\nnh1}\nh(z)\,$
provided by the global section $\,S$) and treat it as a section of the
pull\-back bundle $\,\qs^*\tab$, still denoted by $\,\of$, then, by
(\ref{rem}), 
$\,g_x\w=\of_{\nnh x}^*g_o\w-[\pi^*\nh\ko\of]_x$, and our claim follows since
$\,x\,$ and $\,S\,$ uniquely determine $\,o$.
\end{proof}
We can now prove the following generalization of Theorem~\ref{afifi}.
\begin{theorem}\label{gnrlz}Let there be given a bundle\/ $\,\mf\,$ of
co\-tan\-gent-prin\-ci\-pal bundles with the bundle projections\/
$\,\mf\nh\to\qe\to\bs\,$ as in\/ {\rm(\ref{tbp})}, a tor\-sion-free
connection\/ $\,D\,$ on\/ $\,\bs$, and a vertical metric\/ $\,\vh\,$ on\/ 
$\,\qe$. If\/ $\,g\,$ is a Riemann pull\-back-ex\-ten\-sion metric 
on\/ $\,\mf\nh$, for\/ $\,D,\qe\,$ and\/ $\,\vh$, then the vertical
distributions\/ $\,\dz$ and\/ $\,\vz\hh$ of the bundle projections\/
$\,\mf\nh\to\qe\,$ and\/ $\,\mf\nh\to\bs\,$ are each other's\/ 
$\,g$-or\-thog\-o\-nal complements, $\,\dz$ is\/ $\,g$-null and\/ 
$\,g$-par\-al\-lel, the Le\-vi-Ci\-vi\-ta connection\/ $\,\nabla$ of\/ 
$\,g\,$ is pro\-ject\-a\-ble along both\/ $\,\dz\,$ and\/ 
$\,\vz\nh$, while the\/ $\,\vz\nh$-pro\-ject\-ed tor\-sion-free connection
on\/ $\,\bs$, cf.\ Lemma\/~{\rm\ref{prjcn}}, coincides with\/ $\,D$, and\/
$\,\vh\,$ arises from $\,g\,$ via\/ {\rm(\ref{vrm})}.

Conversely, let the Levi-Civita connection\/ $\,\nabla$ of a
pseu\-\hbox{do\hs-}Riem\-ann\-i\-an metric\/ $\,g\,$ on an\/
$\,n$-di\-men\-sion\-al manifold be pro\-ject\-a\-ble along a null parallel
distribution\/ $\,\dz\hs$ of dimension\/ $\,r$. Then, for every point\/ $\,x\,$
of the manifold, there exist data\/ $\,\mf\nh,\qe,\bs,D,\vh\,$ as above and a
dif\-feo\-mor\-phic identification of a neighborhood\/ of\/ $\,x$ with an
open set in\/ $\,\mf\,$ under which\/ $\,g,\dz\nh$, the connection projected
along\/ $\,\vz\nh=\dz^\perp$, cf.\ Lemma\/~{\rm\ref{prjcn}}, 
and the vertical metric on a local leaf space of\/ $\,\dz\,$ arising from\/
$\,g$ via\/ {\rm(\ref{vrm})} correspond to a Riemann pull\-back-ex\-ten\-sion
metric for\/ $\,D,\qe\,$ and\/ $\,\vh$, the vertical distribution 
of the af\-fine bundle\/ $\,\mf\,$ over\/ $\,\qe$, along with\/ $\,D\,$ and\/
$\,\vh$.
\end{theorem}
\begin{proof}The first part follows from Lemmas~\ref{gnqdf},~\ref{nbprj} and
the final clause of Lemma~\ref{lccoo}: (\ref{cdt}) -- (\ref{ghc}) are just 
(\ref{gab}) -- (\ref{eqv}) for 
$\,B\nnh_{ajk}\w=-2g_{ia}\w\nh\vg_{\hskip-2.6ptjk}^{\hs i}$ and our 
$\,\stf\hn_{jk}\w$.

Under the hypotheses of the second part, Walk\-er's theorem \cite{walker}
applied to $\,\dz$ gives (\ref{gab}) with (\ref{rin}) on a coordinate
neighborhood $\,\,U\,$ of $\,x\,$ which may also be assumed to carry the three
bundle projections as in (\ref{bpv}), the first two of which send
$\,x^{\hh i}\nh,x^{\hh p}\nh,x^{\hh a}$ to $\,x^{\hh i}\nh,x^{\hh p}$ and,
respectively, $\,x^{\hh i}\nh,x^{\hh p}$ to $\,x^{\hh i}\nh$. Now (\ref{gab})
-- (\ref{eqv}) become (\ref{ghc}) with (\ref{cdt}) if one fixes a nonsingular
$\,r\times r\,$ matrix $\,[g_{ia}\w]\,$ of constants and defines 
$\,\stf\hn_{ip}\w,\,\hn\vh_{pq}\w$ and $\,\vg_{\hskip-2.6ptjk}^{\hs i}\,$ by
$\,\stf\hn_{ip}\w=g_{ip}\w$, $\,\vh_{pq}\w=g_{pq}\w$ and
$\,B\nnh_{ajk}\w=-2g_{ia}\w\nh\vg_{\hskip-2.6ptjk}^{\hs i}$. Setting
$\,\xi_{\hh i}\w=g_{ia}\w x^{\hh a}\nh$, we may treat
$\,x^{\hh i}\nh,\xi_{\hh i}\w$ as the local coordinates for $\,\tab$,
associated with the coordinate system $\,x^{\hh i}$ in $\,\bs$, which leads
to the new coordinates $\,x^{\hh i}\nh,x^{\hh p}\nh,\xi_{\hh i}\w$ on
$\,\,U\,$ and, for the corresponding coordinate vector fields, (\ref{hnp})
gives $\,g(\partial/\partial x^{\hh i}\nh,\partial/\partial\xi_{\hn j}\w)
=\delta_{\nnh j}^i$. In other words, each
$\,v=\partial/\partial\xi_{\hh i}\w$, restricted to any each leaf of $\,\dz$,
is the image under the isomorphism (\ref{pxg}) of the $\,1$-form
$\,\xi=dx^{\hh i}$ on $\,\bs$, which realizes $\,\,U\,$ (made smaller, if 
necessary) as an open subset of the total space $\,\mf\,$ of a bundle of 
co\-tan\-gent-prin\-ci\-pal bundles, so as to identify (\ref{bpv}) with a 
restriction version of (\ref{tbp}).

In view of Lemma~\ref{gnqdf}, our $\,g$, given by (\ref{ghc}), is thus a
Riemann pull\-back-ex\-ten\-sion metric for $\,D,\qe\,$ and $\,\vh$, since
the components of $\,D$, the $\,\vz\nh$-pro\-ject\-ed tor\-sion-free
connection, are $\,\vg_{\hskip-2.6ptjk}^{\hs i}$ (see the final clause of 
Lemma~\ref{lccoo}), and $\,\vh$, arising from $\,g\,$ via (\ref{vrm}), 
has the components $\,\vh_{pq}\w=g_{pq}\w$. This completes the proof.
\end{proof}

\end{document}